\documentclass[11pt]{amsart}

\usepackage[margin=1.5in]{geometry}                
\usepackage{amsmath}
\usepackage{amsthm}
\usepackage{amssymb}
\usepackage{enumitem}
\usepackage{placeins}
\usepackage{graphicx}
\usepackage{ amssymb }
\usepackage{epstopdf}
\usepackage[utf8]{inputenc}
\usepackage[english]{babel}
\usepackage{amsthm}
\usepackage{ dsfont }
\usepackage[square,numbers,sort&compress]{natbib} 
\usepackage{amsmath}
\usepackage{amsfonts}
\usepackage{amssymb}
\usepackage{scalerel}[2016/12/29]
\usepackage{bbm}
\usepackage{amsmath,amsfonts,amsthm,amssymb}
\usepackage{fancyhdr}
\usepackage{ amssymb }

\theoremstyle{plain}
\newtheorem{theorem}{Theorem}[section]
\newtheorem{corollary}[theorem]{Corollary}
 
\newtheorem{prop}[theorem]{Proposition}
\newtheorem{lemma}[theorem]{Lemma}

\theoremstyle{definition}
\newtheorem{definition}[theorem]{Definition}
\newtheorem{example}[theorem]{Example} 

\theoremstyle{plain}

\newcommand{\nat}{\mathbb{N}}
\newcommand{\real}{\mathbb{R}}
\newcommand{\N}{\mathbb{N}}

\newcommand{\prob}{\mathbb{P}}
\newcommand{\p}{\mathbb{P}}
\newcommand{\E}{\mathbb{E}}
\newcommand{\expt}{\mathbb{E}}

\newcommand{\floor}[1]{{\left\lfloor #1 \right\rfloor}}

\newcommand{\sset}{\subset}

\newcommand{\al}{\alpha}
\newcommand{\Om}{\Omega}
\newcommand{\mathforall}{\text{ for all }}

\newcommand{\mathand}{\;\text{and}\;}

\newcommand{\ga}{\gamma}

\newcommand{\ep}{\epsilon}
\newcommand{\om}{\omega}

\newcommand{\de}{\delta}

\newcommand{\sig}{\sigma}

\newcommand{\del}{\partial}

\newcommand{\scrM}{\mathcal{M}}

\newcommand{\scrW}{\mathcal{W}}

\newcommand{\scrF}{\mathcal{F}}

\newcommand{\card}[1]{\left\vert #1 \right\vert}

\newcommand{\close}[1]{\mkern 1.5mu\overline{\mkern-1.5mu#1\mkern-1.5mu}\mkern 1.5mu}
\newcommand{\supp}{\text{supp}}
\newcommand{\Z}{\mathds{Z}}
\newcommand{\C}{\mathds{C}}
\newcommand{\R}{\mathds{R}}
\newcommand{\ddd}{\mathellipsis}
\usepackage{MnSymbol}

\DeclareMathOperator*{\argmax}{arg\,max}
\DeclareMathOperator*{\capa}{cap}

\usepackage{yfonts}
\usepackage{ wasysym }

\newcommand{\X}{\times}

\newcommand{\as}{\text{almost surely}}

\newcommand{\smin}{\setminus}
\newcommand{\lf}{\left}
\newcommand{\rg}{\right}

\DeclareMathOperator{\Int}{int}

\author{Duncan Dauvergne}                

\title[Zeros of random polynomials]{A necessary and sufficient condition for convergence of the zeros of random polynomials}

\begin{document}

\maketitle
\begin{abstract}
Consider random polynomials of the form $G_n = \sum_{i=0}^n \xi_i p_i$, where the $\xi_i$ are i.i.d.\ non-degenerate complex random variables, and $\{p_i\}$ is a sequence of orthonormal polynomials with respect to a regular measure $\tau$ supported on a compact set $K$. We show that the zero measure of $G_n$ converges weakly almost surely to the equilibrium measure of $K$ if and only if $\mathbb{E} \log(1 + |\xi_0|) < \infty$. This generalizes the corresponding result of Ibragimov and Zaporozhets in the case when $p_i(z) = z^i$. We also show that the zero measure of $G_n$ converges weakly in probability to the equilibrium measure of $K$ if and only if $\mathbb{P} (|\xi_0| > e^n) = o(n^{-1})$. 

\medskip

Our proofs rely on results from small ball probability and exploit the structure of general orthogonal polynomials. Our methods also work for sequences of asymptotically minimal polynomials in $L^p(\tau)$, where $p \in (0, \infty]$. In particular, sequences of $L^p$-minimal polynomials and (normalized) Faber and Fekete polynomials fall into this class.
\end{abstract}

\section{Introduction}
Let $\{p_0, p_1, \dots\}$ be a sequence of polynomials where each $p_i$ is of degree $i$. Let $\{\xi_j : j \in \nat\}$ be a sequence of i.i.d.\ complex random variables. Throughout the paper, we always assume that the $\xi_j$ are non-degenerate (that is, their distribution is not supported on a single point). We consider the sequence of random polynomials
\begin{equation}
\label{E:Gn}
G_n = \sum_{j=0}^n \xi_j p_j, \qquad n \in \N.
\end{equation}
The polynomial $G_n$ has $D_n$ zeros $z_1, \dots z_{D_n}$, where $D_n = \max \{ j \le n : \xi_j \ne 0\}$. Define the \textbf{zero measure} of $G_n$ by
$$
\mu_{G_n} = \frac{1}{D_n} \sum_{i=1}^{D_n} \de_{z_i}.
$$
We are interested in understanding the global asymptotic behaviour of $\mu_{G_n}$ for various sequences of polynomials $\{p_i\}$. 
In particular, we are interested in finding limits, both almost surely and in probability, of the random measure $\mu_{G_n}$.
Here the underlying topology is the weak topology on probability measures on $\C$.

\medskip

This problem was first considered by Hammersley \cite{hammersley1972zeros} in the context of the \emph{Kac ensemble}
$$
K_n(z) = \sum_{j=0}^n \xi_j z^j, \qquad n \in \nat.
$$
Shortly thereafter, Shparo and Shur \cite{shparo1962} proved the first results about the concentration of the zeros of $G_n$ near the unit circle.

\medskip

The global zero distribution of the Kac ensemble has been extensively studied (see, for example, \cite{hughes2008zeros, ibragimov2013distribution, shparo1962}).
In particular, Ibragimov and Zaporozhets \cite{ibragimov2013distribution} showed that the condition
\begin{equation}
\label{E:IZ}
\expt \log(1 + |\xi_0|) < \infty
\end{equation}
is both necessary and sufficient for $\mu_{K_n}$ to converge weakly almost surely 
to normalized Lebesgue measure $\frac{1}{2\pi} d \theta$ on the unit circle $C = \{z : |z| = 1\}$. 

\medskip

We can view the sequence of monomials $\{1, z, z^2, \dots\}$ used to form the Kac ensemble
as an orthonormal basis for the space $L^2(\frac{1}{2\pi} d \theta)$, 
and the measure $\frac{1}{2\pi} d \theta$ as the equilibrium measure of $C = \supp(\frac{1}{2\pi} d \theta)$. Based on this observation, it is natural to replace $\{1, z, z^2, \dots\}$ with another orthonormal polynomial sequence with respect to a compactly supported measure
$\tau$ to form a sequence as in \eqref{E:Gn}.
We can then ask if the zero measure $\mu_{G_n}$ converges to the equilibrium measure of $\supp(\tau)$.  Note that in the Kac ensemble case, it is a coincidence that the equilibrium measure and the measure $\tau$ are the same.

\medskip

This approach was first taken up by Shiffman and Zelditch in \cite{shiffman1999distribution}.
They proved almost sure convergence results for particular types of measures $\tau$,
in the case where the coefficients $\xi_j$ are i.i.d.\ complex Gaussian random variables. 
Other investigations in this direction have been conducted in \cite{bloom2018universality, bloom2006zeros, bloom2015random, bayraktar2017global, bayraktar2013equidistribution, kabluchko2014asymptotic, pritsker2017equidistribution, bloom2018note, pritsker2018zero}.
In this paper, we continue this line of research. Our primary concern is to find the weakest possible 
conditions on the coefficients $\xi_j$ so that the sequence $\mu_{G_n}$ converges weakly almost surely to the equilibrium measure of $K = \supp(\tau)$, which we denote by $\mu_K$.

\medskip

In particular, we generalize the Ibragimov-Zaporozhets necessity and sufficiency theorem about the Kac ensemble to a wide class of random sums of asymptotically minimal polynomials with i.i.d.\ weights. The asymptotic minimality condition that we use is needed to ensure that the underlying sequence of polynomials is linked in some way to the equilibrium measure of $K$, see the discussion after Definition \ref{D:asym-ext}. This class includes random sums of orthonormal polynomials with i.i.d.\ weights where the background measure $\tau$ satisfies a weak density condition (regularity) on its support. We also analyze when the sequence $\mu_{G_n}$ converges weakly in probability, and prove necessity-sufficiency statements in that case.

\medskip

The specific problem of finding weak convergence conditions for the zeros of random sums of orthonormal or asymptically minimal polynomials has been tackled previously by various authors. The following papers are those with results most directly related to our work.

\medskip

Kabluchko and Zaporozhets \cite{kabluchko2014asymptotic} proved convergence of $\mu_{G_n}$ to $\mu_K$
in probability in the case where the $p_j$ are of the form
$a_j z^j$ for certain sequences $a_j \in \C$ under the Ibragimov-Zaporozhets condition \eqref{E:IZ} and proved almost sure convergence in some special cases. Their results include the case of random sums of orthonormal polynomials when the measure $\tau$ is rotationally symmetric. Their results also address more general random analytic functions, including random orthogonal polynomial arrays generated from a circularly symmetric measure and a circularly symmetric weight function.

\medskip

Bloom and Dauvergne \cite{bloom2018universality} extended this result by showing that $\mu_{G_n}$ converges to $\mu_K$ 
in probability for any sequence of polynomials $\{p_j\}$ generated 
from a measure $\tau$ satisfying the Bernstein-Markov property (we will discuss this property later in the introduction)
under the condition
\begin{equation}
\label{E:BD-moment}
\prob(|\xi_0| > e^n) = o(n^{-1}).
\end{equation}
This condition is slightly weaker than \eqref{E:IZ}. 
They also established almost sure convergence of $\mu_{G_n}$ to $\mu_K$ in the case of a rotationally symmetric measure $\tau$ and weights satisfying \eqref{E:IZ}.

\medskip

In \cite{pritsker2018zero}, Pritsker considered random sums of polynomials with asymptotically minimal $L^\infty$-norm on a compact set $K$. He established almost sure convergence of $\mu_{G_n}$ to $\mu_K$ when the i.i.d.\ coefficients $\xi_i$ satisfy the moment conditions
\begin{equation}
\label{E:log-cond}
\expt \log (1 + |\xi_0|) < \infty \qquad \mathand \quad \sup_{z \in \C} \expt(\log^- |\xi_0 + z|)^t < \infty \;\;\;\; \text{for some $t > 1$.}
\end{equation}
Here $f^- = - \min (0, f)$. Pristker's results also hold when $K$ is simply connected with empty interior when the second condition above is replaced by 
$
\expt \log^- |\xi_0| < \infty.
$

\medskip

Pritsker and Ramachandran \cite{pritsker2017equidistribution} have studied this problem for sequences $\{p_j\}$ given by Faber, Bergman, or Szeg\H{o} polynomials where the compact set $K$ is the closure of a Jordan domain with an analytic boundary $D$. In this case they showed that condition \eqref{E:IZ} was both necessary and sufficient for almost sure convergence of the zero measure $\mu_{G_n}$. All these classes of polynomials fall into the framework of this paper.

\medskip

In the papers discussed above, almost sure convergence results are only obtained when either the underlying polynomials have a particular explicit structure which makes them easier to analyze, or else the random variables satisfy particular anticoncentration estimates (i.e. \eqref{E:log-cond}). In contrast, our proofs only require estimates that can be verified for broad classes of polynomials and estimates from small ball probability that hold for all non-degenerate random variables. This allows us to prove optimal convergence theorems in great generality. Moreover, our proofs of necessity and sufficiency go through in the same way for both convergence in probability and almost sure convergence. In much of the previous work, different types arguments were needed to establish the two different types of convergence despite the fact that the underlying phenomena are quite similar.

\subsection{Main Results}
\label{SS:main-results}

\medskip

Let $\tau$ be a probability measure with non-polar compact support $K \sset \C$ (note that if $K$ is polar, then the equilibrium measure is not even uniquely defined).
Let $\{p_n = \sum_{i=0}^n a_{n, i} z^i \}$ be the sequence of orthonormal polynomials formed by applying the Gram-Schmidt
procedure to the sequence $\{1, z, z^2, \dots\}$. We say that the measure $\tau$ is \textbf{regular} on its support $K$ if
$$
\lim_{n \to \infty} \frac{1}n \log |a_{n, n}| = -\log \capa(K).
$$
Regularity of a measure is a very weak density condition that is satisfied by almost any natural measure (see the discussion in Section \ref{SS:regularity}). We first state our main convergence theorems for random sums of orthonormal polynomials.

\begin{theorem}(Almost Sure Convergence)
	\label{T:extremal-as} 
	Let $\tau$ be a regular measure with non-polar compact support $K \sset \C$. Let $\{p_i\}$ be the sequence of orthonormal polynomials in $L^2(\tau)$ formed by the Gram-Schmidt procedure, and let $\{\xi_i\}$
	be a sequence of i.i.d.\ non-degenerate complex random variables.
	Define the random orthonormal polynomial
	$$
	G_n = \sum_{i=0}^n \xi_i p_i,
	$$
	and let $\mu_{G_n}$ be the zero measure of $G_n$. 
	
	\medskip
	
	The measure $\mu_{G_n}$ converges weakly almost surely to the equilibrium measure $\mu_K$ if and only if 
	\begin{equation}
	\label{E:as-cond}
	\expt \log(1 + |\xi_0|) < \infty.
	\end{equation}
	In particular, if \eqref{E:as-cond} fails, then the sequence $\mu_{G_n}$ has no almost sure limit in the space of probability measures on $\C$.
\end{theorem}

\begin{theorem}(Convergence in Probability)
	\label{T:extremal-p}
	Let $G_n$ be as in Theorem \ref{T:extremal-as}.
	The measure $\mu_{G_n}$ converges weakly in probability to $\mu_K$ if and only if 
	\begin{equation}
	\label{E:prob-cond}
	\prob(|\xi_0| > e^n) = o(n^{-1}).
	\end{equation}
	Again, if \eqref{E:prob-cond} fails, then the sequence $\mu_{G_n}$ has no limit in probability in the space of probability measures on $\C$.
\end{theorem}

Theorems \ref{T:extremal-as} and \ref{T:extremal-p} are special cases of our main convergence result about random asymptotically minimal polynomials. 

\begin{definition}
	\label{D:asym-ext} A sequence of degree-$n$ polynomials $\{p_n = \sum_{i=0}^n a_{n, i} z^i : n \in \N\}$ is \textbf{asymptotically minimal} on a compact set $K \sset \C$ if there exists a regular measure $\tau$ with $\supp(\tau) = K$ and a $p \in (0, \infty]$ such that
	$$
	\lim_{n \to \infty} \frac{1}n \log |a_{n, n}| = - \log \capa(K) \qquad \mathand \qquad \lim_{n \to \infty} \frac{1}n \log \|p_n\|_{L^p(\tau)} = 0.
	$$
	Note that in the case when $p = \infty$, we can drop the need for a measure $\tau$.
\end{definition}

In the above definition, the term asymptotically minimal comes from the fact that the ratio $\frac{1}n \log (\|p_n\|_{L^p(\tau)}/|a_{n, n}|)$ converges to the minimal value $\log \capa (K)$. Either this minimality condition or something similar is essentially necessary for the sequence $\{p_n\}$ to be linked to the equilibrium measure of $K$. Note that the orthonormal sequences of polynomials $\{p_n\}$ in Theorems \ref{T:extremal-as} and \ref{T:extremal-p} are truly minimal, rather than just asymptotically minimal: for every $n$, the polynomial $p_{n, n}/a_{n, n}$ minimizes the $L^2(\tau)$-norm among all degree-$n$ monic polynomials.

\medskip

We require the convergence of both the leading coefficients and the $L^p(\tau)$-norms (rather than just the ratio) so that different terms are of comparable size when we take a random sum of asymptotically minimal polynomials. If only the ratio $\frac{1}n \log (\|p_n\|_{L^p(\tau)}/|a_{n, n}|)$ converges, then the corresponding random sum of asymptotically minimal polynomials need not have any structure. For example, for any sequence $\{a_{n, n}\}$, the polynomials $\{p_n = a_{n, n} z^n \}$ satisfy convergence of the ratio $\frac{1}n \log (\|p_n\|_{L^p(\tau)}/|a_{n, n}|)$, but we can choose the numbers $\{a_{n, n}\}$ so that the zero measure of randomized sum $\sum_{i=0}^n \xi_i p_i$ will not converge for any choice of the distribution of $\xi_0$.

\medskip

Since regular measures are defined with respect to orthonormal polynomials, the $L^2$ norm seems to play a special role in the above definition. However, this is not the case. Let $\scrM_n$ be the set of monic, degree-$n$ polynomials. By Theorem 3.4.1 from \cite{stahl1992general}, if
\begin{equation}
\label{E:p-regular}
\lim_{n \to \infty} \frac{1}n \log ( \inf_{q \in \scrM_n} ||q||_{L^p(\tau)} ) = \log \capa K
\end{equation}
for one $p \in (0, \infty)$, then it holds for all $p \in (0, \infty)$. As discussed above, the minimizing sequence of polynomials in \eqref{E:p-regular} with $p = 2$ is the renormalized sequence of orthonormal polynomials, so we can take Equation \eqref{E:p-regular} to be the definition of regularity of $\tau$ when $p = 2$. Hence a measure is `$2$-regular' if and only if it is `$p$-regular' (in the sense of Equation \eqref{E:p-regular}) for all $p \in (0, \infty)$. Before moving on, we remark that \eqref{E:p-regular} implies that for any sequence of degree-$n$ polynomials $p_n(z) = a_{n, n} z^n + \dots, n \in \N$ and any regular measure $\tau$,
\begin{equation}
\label{E:lower-lower}
\liminf_{n \to \infty} \frac{1}{n} \log \lf(\frac{||p_n||_{L^p(\tau)}}{|a_{n, n}|} \rg) \ge \log \capa(K).
\end{equation}
This will useful to keep in mind later on.

\medskip

We can now state the more general version of Theorems \ref{T:extremal-as} and \ref{T:extremal-p}.

\begin{theorem}
	\label{T:extremal}
	Let $K$ be a non-polar compact subset of $\C$. Theorem \ref{T:extremal-as} and Theorem \ref{T:extremal-p} still hold when the sequence $\{p_i : i \in \N\}$ is any sequence of asymptotically minimal polynomials on $K$. 
\end{theorem}

As discussed above, the sufficiency of condition \eqref{E:prob-cond} for Theorem \ref{T:extremal-p} was proven as Theorem 5.3 in \cite{bloom2018universality}
for Bernstein-Markov measures $\tau$ with regular support $K$. These proofs can be extended to include all cases of Theorem \ref{T:extremal-p} with a few modifications. The proof ideas can also be used to show the convergence in probability in Theorem \ref{T:extremal} when the sequence $\{p_n\}$ has an additional assumption about root concentration or speed of convergence (i.e. see \cite{bloom2018note}).

\medskip

However, the method of \cite{bloom2018universality} does not extend as easily to asymptotically minimal polynomials without any condition on root concentration or speed of convergence, and does
not extend at all to the case of almost sure convergence. 

\medskip

Finally, in Theorem \ref{T:extremal}, the condition of non-degeneracy on the random variables is crucial. The asymptotic structure of the roots can change when the coefficients are degenerate. For example, consider the case when $p_0 = 1$ and $p_n = z^n - z^{n-1}$ for $n \ge 1$. This is a sequence of asymptotically minimal polynomials for the unit circle $C$, but the zero measure of the deterministic sum $\sum_{i=0}^n p_i(z) = z^n$ converges to a $\de$-mass at $0$, rather than the equilibrium measure on $C$. For a less contrived example that does quite not fit into our framework here, see \cite{kabluchko2014asymptotic}, Figure 2.

\subsection{Examples}

In addition to orthonormal polynomials, many other natural classes of polynomials fit into the framework of Theorem \ref{T:extremal}. We give three examples here.

\begin{example}[$L^p$-minimal polynomials and Chebyshev polynomials]
	\label{E:Lp-extremal}
	Let $K$ be a non-polar compact set. Fix a $p \in (0, \infty]$ and a regular measure $\tau$ on $K$. Let $r_n$ be a monic degree-$n$ polynomial on $K$ with minimal $L^p(\tau)$ norm. Note that $r_n$ exists by a compactness argument, but is not necessarily unique when $p < 1$ (see discussion on pg. 84, \cite{stahl1992general}). Then we have
	\begin{equation}
	\label{E:p-regular-2}
	\lim_{n \to \infty} \frac{1}n \log ||r_n||_{L^p(\tau)} = \log \capa (K).
	\end{equation}
	In particular, the normalized sequence $\{q_n = r_n/||r_n||_{L^p(\tau)} : n \in \N\}$ is asymptotically minimal on $K$. 
\end{example}

For $p \in (0, \infty)$, Equation \eqref{E:p-regular-2} follows from the discussion above about regularity. For the Chebyshev polynomial case $p = \infty$, the measure $\tau$ plays no role and Equation \eqref{E:p-regular-2} holds for any non-polar  compact set $K$ (see \cite{tsuji1959potential}, Theorem III.26).

\begin{example}[Fekete polynomials] 
	\label{E:fekete}
	Fix a non-polar compact set $K$. Let $z_{n, 1}, \dots, z_{n, n}$ be Fekete points in $K$ (i.e. a set of $n$ points which maximize the Vandermonde determinant $\Pi_{1 \le i < j \le n} |z_{n, i} - z_{n, j}|$). Let 
	$p_n(z) = \ga_n \Pi_{j=1}^n(z - z_{n, j})$, where $\ga_n \in \R^+$ is a normalizing constant chosen so that $p_n$ has uniform norm equal to $1$ on $K$. Then the sequence $\{p_n\}$ is asymptotically minimal on $K$ (see Theorem 5.4.4, \cite{ransford1995potential}).
\end{example}

\begin{example}[Faber polynomials] 
	\label{E:faber}
	Let $K \sset \C$ be compact and assume that $(\C \smin K) \cup \{\infty\}$ is simply connected in $\C \cup \{\infty\}$. We can define the Faber polynomials $f_n$ on $K$ as follows. Let
	$$
	\Phi:(\C \smin K) \cup \{\infty\} \to (\C \smin \{z : |z| \le 1\}) \cup \{\infty\}
	$$
	be the Riemann mapping satisfying $\Phi(\infty) = \infty$ and $\Phi'(\infty) > 0$. We can write
	$$
	\Phi(z) = \frac{z}{\capa(K)} + a_0 + \frac{a_1}{z} + \dots.
	$$
	Define the $n$th Faber polynomial $f_n(z)$ as the unique polynomial of the form $\Phi^n - g_n$, where $g_n$ contains only negative powers of $z$. After renormalizing $f_n$ to have uniform norm equal to $1$ on $K$, the sequence $f_n$ is asymptotically minimal (see, for example, \cite{levenberg2018zeros}).
\end{example}

\subsection{Remarks about the regularity of $\tau$ and related conditions}
\label{SS:regularity}
(i) Regularity of a measure $\tau$ on $\supp(\tau) = K$ is essentially a statement about how dense the measure $\tau$ is on $K$. Almost all naturally arising measures are regular. For example, when $K$ is regular (that is, the outer boundary of $K$ is regular for the Dirichlet problem), regularity of $\tau$ is implied by the following rather weak density condition (see Proposition 3.3, \cite{bloom2015bernstein}): there exists an $r_0 > 0$ and $t > 0$ such that for any $z$ in the outer boundary of $K$ and any $r < r_0$, we have that
$$
\tau(B(z, r)) \ge r^t.
$$
(ii) A measure $\tau$ with support $K$ has the \textbf{Bernstein-Markov property} if for every $\ep > 0$, there exists a constant $c > 0$ such that for any polynomial $p$ of degree $n$, we have that
$$
||p||_{K} \le ce^{\ep n} ||p||_{L^2(K)},
$$
where $||\cdot||_K$ is the uniform norm on $K$.
Any Bernstein-Markov measure is regular (see discussion on p. 67, \cite{stahl1992general}). Also, for a regular compact set $K$, regularity of $\tau$ and the Bernstein-Markov property are equivalent (see Theorem 3.2.3, \cite{stahl1992general}). Much of the previous work on complex zeros of random sums of orthonormal polynomials has focussed on the case when $\tau$ has regular support (i.e. \cite{bloom2018universality, bayraktar2013equidistribution}), where the distinction between regular and Bernstein-Markov measures is unnecessary. 

\subsection*{The multivariable and weighted cases}

The global zero distribution of random polynomials has also been studied in other contexts. 
In particular, much of the potential theory used in the study of one-variable random polynomials can be adapted to the multivariable setting (see, for example, \cite{bloom2006zeros, bayraktar2013equidistribution, bayraktar2017global, bloom2007random, bloom2015random}). 
However, the properties of asymptotically minimal polynomials and orthonormal polynomials used for the proofs in this paper do not have obvious multivariable equivalents. 
It is thus unclear if the methods we employ here can be fully extended to this context. 

\medskip

Another interesting extension is to the case of random sums of orthonormal polynomials in the presence of an external field (i.e. random sums of the first $n$ orthonormal polynomials in $L^2(w^{2n}\tau)$ for a weight function $w:\C \to [0, \infty)$ and a compactly supported probability measure $\tau$). This case has been analyzed when the underlying polynomials have circular symmetry and hence can be computed somewhat explicitly (i.e. see \cite{bayraktar2017global, bloom2018universality,  kabluchko2014asymptotic}), but understanding the general case here seems quite difficult.

\subsection*{Organization of the paper and a brief discussion of the proofs}

In Section \ref{S:prelim}, we introduce the necessary background for the paper. In Section \ref{S:det-cvg}, we prove a deterministic convergence criterion for zero measures (up to one result, which we leave to the appendix). Letting $\tau$ be regular measure on a compact set $K$, $p \in (0, \infty]$, and   $\{p_n = \sum_{i=0}^n a_{n, i} z^i\}$, our deterministic criterion allows us to prove convergence of the zero measures $\mu_{p_n} \to \mu_K$ whenever both the ratio $||p_n||_{L^p(\tau)}/|a_{n, n - i_n}|$ is asymptotically minimal for some sequence $i_n = o(n/\log n)$, and the values of $p_n$ on the interior of the polynomially convex hull of $K$, $\Int (P(K))$, are sufficiently large in modulus.
\medskip

Similar criteria have appeared before when $p_n$ is renormalized by the leading coefficient $a_{n, n}$ instead of a close-to-leading coefficient $a_{n, n - i_n}$ (see, for example, \cite{bloom2007random}, Theorem 1.2 or \cite{blatt1988jentzsch}, pp. 309-310). However, the random polynomials $G_n$ that we consider may have leading coefficients which are extremely small with a non-negligible probability (i.e. if $\expt \log^- |\xi_0| = \infty$) so theorems that only allow renormalization by the leading coefficient cannot be applied to our case.

\medskip

In Section \ref{S:sufficiency}, we use this deterministic criterion to prove the sufficiency of the conditions \eqref{E:as-cond} and \eqref{E:prob-cond} for Theorem \ref{T:extremal}. To check the criterion, the key step is using a small ball probability estimate adapted from a result of Nguyen and Vu \cite{nguyen2011optimal} to show that the values of $G_n$ on $\Int (P(K))$ are sufficiently large.

\medskip

This estimate was also used in \cite{bloom2018universality} to establish convergence of $\mu_{G_n}$ 
in the case when $p_j(z) = a_j z^j$. To apply the Nguyen-Vu result it suffices to show that for almost every $z \in \Int(P(K))$, 
the sequence $\{p_0(z), p_1(z), \dots\}$ is sufficiently spread out. This is done via Cartan's estimate and understanding the structure of the polynomials $\{p_i\}$.

\medskip

In Section \ref{S:necessity}, we prove the necessity statements in Theorem \ref{T:extremal}. To do this, we show that if the random variables $\xi_i$ fail to satisfy the required moment conditions, then one term $\xi_j p_j$ for $j \in [n/4, n/2]$ in the sum for $G_n$ dominates the others with an asymptotically non-negligible probability. By Rouch\'e's theorem, this forces at least half of the zeros of $G_n$ to lie outside a disk of arbitrarily large radius. This will imply that the zero measures $\mu_{G_n}$ have no weak limit.

\medskip

A similar idea was used to prove sufficiency in \cite{ibragimov2013distribution} (see also \cite{pritsker2017equidistribution}). However, unlike in those papers, our proof requires very little knowledge of the structure of the polynomials and relies mostly on understanding of the sequence of random variables $\{\xi_i\}.$

\subsection*{Notation}
Throughout the paper, we use the following notation. For a compact set $K$, let $U(K)$ be the unbounded component of $K^c$, and let $P(K) = U(K)^c$ be the polynomially convex hull of $K$ (i.e. $K$ with its holes filled in). We write $\del P(K)$ for the outer boundary of $K$. All sequences of deterministic polynomials we consider will be written as $\{p_n = \sum_{i=0}^n a_{n, i} z^i\}$. We will let $\mu_n$ denote the zero measure of $p_n$. The polynomial $G_n = \sum_{i=0}^n \xi_i p_i$ will always refer to a random polynomial where $\{p_n\}$ is a sequence of asymptotically minimal polynomials, and the sequence $\{\xi_i\}$ is an i.i.d.\ sequence of complex non-degenerate random variables. We will also write $G_n = \sum_{i=0}^n \zeta_{n, i} z^i$ for the decomposition of $G_n$ into a sum of monomials. The coefficient $\zeta_{n,i} = \sum_{j=i}^n \xi_j a_{j, i}$.

\medskip

As discussed at the beginning of the introduction, $G_n$ is of degree $D_n = \sup \{i \le n :  \xi_i \ne 0 \}$, and the leading coefficient is $\xi_{D_n} a_{D_n, D_n}$. However, the non-degeneracy of $\xi_i$ guarantees that $D_n/n \to 1$ almost surely, and no parts of the proofs are affected by this discrepancy between $D_n$ and $n$. Hence we will treat $G_n$ as if it were always degree $n$ in order to avoid carrying the $D_n$ notation throughout the paper.

\section{Preliminaries}
\label{S:prelim}

\subsection{Convergence of random measures}

For a random polynomial $G_n$, we will often write $G_n(z, \om)$, where $\om$ is a point in a background probability space $\Om$. We do this so we can more easily distinguish between two types of almost sure statements: one for almost every $\om \in \Om$, and one for (Lebesgue)-almost every $z$ in some subset of $\C$.

\medskip

Recall that a sequence of probability measures $\mu_n$ on $\C$ converges weakly to $\mu$ if for every continuous, bounded function $f: \C \to \R$, we have that
$$
\lim_{n \to \infty} \int_\C f d \mu_n = \int_\C f d \mu.
$$
The random measures $\mu_{G_n} \to \mu$ almost surely if $\mu_{G_n}(\om)$ converges weakly to $\mu$ for almost every $\om \in \Om$. The random measures
$
\mu_{G_n} \to \mu
$
in probability if for every weakly open set $\mathcal{O}$ containing $\mu$, we have that
$$
\lim_{n \to \infty} \prob(\mu_{G_n} \in \mathcal{O}) = 1.
$$
Equivalently, $\mu_{G_n} \to \mu$ in probability if for every subsequence $J \sset \N$, there is a further subsequence  $J_0 \sset J$ such that $\{\mu_n : n \in J_0\}$ converges to $\mu$ almost surely.

\subsection{Potential Theory}
\label{SS:pot-theory}
Let $D \sset \C \cup \{\infty\}$ be an open set. A function $u:D \to [-\infty, \infty)$ is \textbf{subharmonic} on $D$ if $u$ is upper semicontinuous, and if $u$ satisfies the sub-mean inequality. That is, for every $z \in D$ there exists a radius $\rho > 0$ such that
$$
u(z) \le \frac{1}{2\pi} \int_0^{2\pi} u(z + re^{it})dt, \qquad \text{for all } r \in (0, \rho].
$$
If $f$ is analytic on $D$ then $\log |f|$ is subharmonic.

\medskip

Now let $K \sset \C$ be a compact set and let $\mu$ be a probability measure on $K$. Define the \textbf{logarithmic potential} of $\mu$ by the formula
$$
p_\mu(z) = -\int \log |z - x| d\mu(x).
$$
For a probability measure $\mu$, the function $-p_\mu$ is always subharmonic on $\C$, and harmonic outside of $\supp(\mu)$. In particular, if $\mu$ is the zero measure of a degree-$n$ monic polynomial $q$, then 
\begin{equation*}
\frac{1}n \log |q(z)| = - p_\mu(z).
\end{equation*}
We note here that a measure is uniquely determined by its logarithmic potential (this can be seen by noting that applying the distributional Laplacian to $-p_\mu$ gives the measure $2\pi \mu$).

\medskip

 The \textbf{(logarithmic) energy} of $\mu$ is given by
$$
I(\mu) = \int p_\mu(z) d\mu(z) = - \int \int  \log |z - x| d\mu(x)d\mu(z).
$$
For a compact set $K$, set 
\begin{equation}
\label{E:AK}
A_K = \inf \{ I(\mu) : \mu \text{ is a probability measure on $K$}\}.
\end{equation}
We then define the \textbf{(logarithmic) capacity} of $K$ by
$$
\capa(K) = e^{-A_K}.
$$
A set $K$ is \textbf{polar} if $\capa(K) = 0$ (i.e. $A_K = \infty$). Any polar set has Lebesgue planar measure zero. We say that a property holds \textbf{quasi-everywhere} on a set $D \sset \C$ if it holds for all $z \in D$ outside of a polar set.

\medskip

 Note that any compact set $K$ has finite capacity. When $K$ is a non-polar compact set, there is a unique probability measure $\mu_K$ on $K$ that attains the infimum in \eqref{E:AK}. This is known as the \textbf{equilibrium measure}. It is always supported on the outer boundary $\del P(K)$, and hence its potential is harmonic on both $U(K)$ (the unbounded complement of $K$) and on $\Int (P(K))$ (the interior of the polynomially convex hull of $K$). We will use the following characterization of $\mu_K$.
 
 \begin{theorem}[see Theorem A.1, \cite{stahl1992general}]
\label{T:eq-def}
Let $K$ be a non-polar compact subset of $\C$. Then $\mu_K$ satisfies
\begin{enumerate}[label=(\roman*)]
\item $p_{\mu_K}(z) \le - \log \capa(K) \quad \text{for } z \in \C.$
\item $
p_{\mu_K}(z) = - \log \capa(K) \quad \text{for every } z \in \Int P(K) \text{ and quasi-every } z \in \del P(K).
$
\end{enumerate}
Moreover, $\mu_K$ is the only measure satisfying these properties.
\end{theorem}

As the equilibrium measure is an essential object of study for us, we restrict our attention to non-polar compact sets for the remainder of the paper.

\section{A deterministic convergence statement}
\label{S:det-cvg}

In this section, we prove the following deterministic result about the convergence of zero measures of polynomials to an equilibrium measure on a compact set, up to one fact which is left to the appendix. We will also prove a few facts about potentials along the way that will be used later in the paper.

\begin{theorem}
\label{T:det-cvg-in}
Let $\{p_n = \sum_{k=0}^n a_{n, k} z^k\}$ be a sequence of degree-$n$ polynomials and let $K \sset \C$ be a non-polar compact set. Suppose that there exists a sequence $i_n = o(\frac{n}{\log n})$ such that the following conditions hold:
\begin{enumerate}[label=(\roman*)]
\item There exists a regular probability measure $\tau$ on $K$ and a $p \in (0, \infty]$ such that
$$
\limsup_{n \to \infty} \frac{1}{n} \log \lf(\frac{||p_n||_{L^p(\tau)}}{|a_{n, n - i_n}|} \rg) \le \log \capa(K).
$$
\item For almost every $z \in \Int(P(K))$, we have that
$$
\liminf_{n \to \infty} \frac{1}{n} \log \lf( \frac{|p_n(z)|}{|a_{n, n - i_n}|} \rg) \ge \log \capa(K).
$$
\end{enumerate}
Then $\mu_n$ converges weakly to $\mu_K$.

\end{theorem}

Note that the second condition is vacuously true if $P(K)$ has empty interior. The proof of Theorem \ref{T:det-cvg-in} consists of two parts. We first show that the theorem holds with $i_n = 0$ for all $n$. This is done in Section \ref{SS:in-0}. We then extend this to all $i_n$ in Section \ref{SS:in-all} by showing how the zero measures of polynomials in the case of general $i_n$ can be related to zero measures of polynomials satisfying the conditions of Theorem \ref{T:det-cvg-in} with $i_n = 0$.

\subsection{The $i_n = 0$ case of Theorem \ref{T:det-cvg-in}}
\label{SS:in-0}

To prove Theorem \ref{T:det-cvg-in}, we first show the following.
\begin{theorem}
\label{T:det-cvg}
Let $\{p_n\}$ be a sequence of degree-$n$ polynomials and let $K \sset \C$ be a non-polar compact set. Suppose that the following conditions hold:
\begin{enumerate}[label=(\roman*)]
\item There exists a regular probability measure $\tau$ on $K$ and a $p \in (0, \infty]$ such that
$$
\lim_{n \to \infty} \frac{1}{n} \log \lf(\frac{||p_n||_{L^p(\tau)}}{|a_{n, n}|} \rg) = \log \capa(K).
$$
\item For almost every $z \in \Int(P(K))$, we have that
$$
\liminf_{n \to \infty} \frac{1}{n} \log \lf( \frac{|p_n(z)|}{|a_{n, n}|} \rg) \ge \log \capa(K).
$$
\end{enumerate}
Then the zero measures $\mu_n$ of $p_n$ converge weakly to $\mu_K$.

\end{theorem}

Theorem \ref{T:det-cvg} is essentially Theorem \ref{T:det-cvg-in} in the special case when $i_n = 0$. The only difference between the two statements is in the apparent strengthening of condition (i) in Theorem \ref{T:det-cvg}. However, this is only a superficial difference; if condition (i) of Theorem \ref{T:det-cvg-in} holds with $i_n = 0$, then condition (i) of Theorem \ref{T:det-cvg} also holds by \eqref{E:lower-lower}.

\medskip

The first step needed to prove Theorem \ref{T:det-cvg} is to show that the sequence of zero measures $\mu_n$ is tight. 

\begin{lemma}
\label{L:tightness}
Let $p_n$ be a sequence of degree-$n$ polynomials, let $K \sset \C$ be a non-polar compact set, and suppose that assumption (i) of Theorem \ref{T:det-cvg} holds. 
\begin{enumerate}[label=(\roman*)]
\item The sequence $\mu_n$ is tight, and any subsequential limit $\mu$ of $\mu_n$ is supported in $P(K)$. 
\item Let $V \sset U(K)$ be any closed set. Let $x_{n, 1}, \dots, x_{n, \ell(n)}$ be the roots of $p_n$ in $V$. Then
$$
\limsup_{n \to \infty} \frac{1}n \log \prod_{i=1}^{\ell(n)} |x_{n, i}| \le 0.
$$
\end{enumerate}
\end{lemma}

To prove Lemma \ref{L:tightness}, we use the following result from \cite{stahl1992general}.

\begin{lemma} (Lemma 1.3.2 from \cite{stahl1992general})
\label{L:move-in-roots}
Let $S$ be a compact subset of $\C$, and let $U$ be the unbounded component of $S^c$. Then for any closed set $V \sset U$, there exists $a < 1$ and $k \in \nat$ such that for any points $x_1, \ddd, x_k \in V$, there exist $k$ points $y_1, \ddd y_k \in P(S)$ for which the rational function
$$
r_k(z) = \prod_{j=1}^k \frac{z - y_j}{z - x_j}
$$
satisfies the inequality $||r_k||_S \le a$ (here $||\cdot||_S$ is the uniform norm on $S$). In particular, when $V$ is outside of the convex hull of $S$, we may take $k = 1$.
\end{lemma}

Note that in \cite{stahl1992general}, the above lemma is stated for the case $V$ compact. However, the proof goes through for all closed $V$.

\begin{proof}[Proof of Lemma \ref{L:tightness}] Without loss of generality, we may assume that $a_{n, n} = 1$. Let $p \in (0, \infty]$ and $\tau$ be given by assumption (i) of Theorem \ref{T:det-cvg}, and let $\{q_n : n \in \N\}$ be the sequence of minimal monic polynomials in $L^p(\tau)$, as in Example \ref{E:Lp-extremal}. Assumption (i) and the regularity of $\tau$ implies that
\begin{equation}
\label{E:q-p}
\lim_{n \to \infty} \frac{1}n \log \lf(\frac{||p_n||_{L^p(\tau)}}{||q_n||_{L^p(\tau)}} \rg) = 0.
\end{equation}
Now define $P(K)_\ep = \{z \in \C : d(z, P(K)) < \ep\}$. To prove both the tightness and subsequential limit claims in (i), it is enough to show that for every $\ep > 0$,
$$
\lim_{n \to \infty} \mu_n(P(K)^c_\ep)=0.
$$
Equivalently, it is enough to show that the number of roots $\ell(n)$ of $p_n$ in $P(K)^c_\ep$ satisfies $\ell(n) = o(n)$.
 Let $x_{n, 1}, \dots, x_{n, \ell(n)}$ denote the roots of $p_n$ in $P(K)^c_\ep$. Let $a < 1, k \in \N$ be as in Lemma \ref{L:move-in-roots} for the sets $V = P(K)^c_\ep$ and $S = K$. For every $i \in \{0, \dots, \floor{\ell(n)/k} - 1\}$, Lemma \ref{L:move-in-roots} implies that we can find points $y_{n, ki + 1}, \dots, y_{n, ki + k} \in P(K)$ such that
\begin{equation*}
\lf \| \prod_{j=1}^k \frac{z - y_{n, ki + j}}{z - x_{n, ki + j}} \rg\|_K \le a.
\end{equation*}
Therefore
\begin{equation}
\label{E:pn-prod}
\begin{split}
\lf \|p_n(z)\prod_{i=1}^{k\floor{\ell(n)/k}} \frac{z - y_{n, i}}{z - x_{n, i}} \rg\|_{L^p(\tau)} &\le \|p_n\|_{L^p(\tau)} \prod_{i=0}^{\floor{\ell(n)/k} - 1} \lf \| \prod_{j=1}^k \frac{z - y_{n, ki + j}}{z - x_{n, ki + j}} \rg\|_K \\
&\le \|p_n\|_{L^p(\tau)} a^{\floor{\ell(n)/k}}.
\end{split}
\end{equation}
As the polynomial on the left hand side above is monic and degree $n$, the minimality of $q_n$ implies that the left hand side of \eqref{E:pn-prod} is bounded below by $||q_n||_{L^p(\tau)}.$ By equation \eqref{E:q-p}, this implies that $\ell(n) = o(n)$, completing the proof of (i).

\medskip

For (ii), observe that since $V$ and $P(K)$ are disjoint closed sets and $P(K)$ is bounded, there exist a constant $c > 0$ such that
$$
\frac{|z - y|}{|z - x|} \le \frac{c}{|x|} \qquad \text{ for every } z, y \in P(K) \mathand x \in V.
$$ 
In particular, this holds for $x = x_{n, i}$ and $y = y_{n, i}$ in \eqref{E:pn-prod}, implying that the left hand side of \eqref{E:pn-prod} is bounded above by
$$
||p_n||_{L^p(\tau)} c^{\ell(n)} \prod_{i=1}^{\ell(n)}|x_{n, i}|^{-1}.
$$ 
Since the left hand side of \eqref{E:pn-prod} is also bounded below by $||q_n||_{L^p(\tau)}$, Equation \eqref{E:q-p} implies that
$$
\lim_{n \to \infty} \frac{1}n \log \lf(c^{-\ell(n)} \prod_{i=1}^{\ell(n)}|x_{n, i}| \rg) \le 0.
$$
Using that $\ell(n) = o(n)$ (proved as part (i)) to remove the $c^{-\ell(n)}$ term proves (ii).
\end{proof}

We also need a version of the principle of descent which we will use repeatedly in the paper to bound logarithmic potentials. This version of the principle of descent is specific to zero measures arising from sequences of polynomials satisfying assumption (i) of Theorem \ref{T:det-cvg}, but eliminates the need for $\{\mu_n\}$ to have a common compact support. 

\begin{lemma}
\label{L:p-descent}
Let $\{p_n\}$ be a sequence of polynomials satisfying assumption (i) of Theorem \ref{T:det-cvg}. Assume that the sequence of zero measures $\mu_n$ converges weakly to a measure $\mu$. Then for any $z \in \C$, and any sequence $z_n \to z$, we have that
$$
\liminf_{n \to \infty} p_{\mu_n}(z_n) \ge p_\mu(z).
$$
\end{lemma}

\begin{proof} We can assume that $a_{n, n} = 1$ for all $n$.
Let $z \in \C$, and let $z_n \to z$. There exists some $r > 0$ such that the compact set $K$, the sequence $\{z_n\}$ and the point $z$ all lie in the disk $D_r = \{z \in \C : |z| \le r\}$. To prove the convergence statement for this sequence, we first approximate $\mu_n$ by a sequence of probability measures supported on $D_{2r}$. Let $y_{n, 1}, \dots, y_{n, k(n)}$ be the zeros of $p_n$ (with multiplicity) in $D_{2r}$ and let $y_{n, k(n) + 1}, \dots, y_{n, n}$ be the roots outside of $D_{2r}$. Define $q_n$ to be the monic degree-$n$ polynomial whose roots are $y_{n, 1}, \dots, y_{n, k(n)}$, with $n - k(n)$ roots at $z_n + 1$. Let $\nu_n$ be the zero measure of $q_n$.

\medskip

By Lemma \ref{L:tightness}, the measure $\mu$ is supported on $K \sset D_{r}$, so $n - k(n) = o(n)$ and hence $\nu_n \to \mu$ weakly. Hence by the usual principle of descent (see Appendix A.III in \cite{stahl1992general})
$$
\liminf_{n \to \infty} p_{\nu_n} (z_n) \ge p_{\mu}(z).
$$
Therefore it is enough to show that
\begin{equation}
\label{E:pot-close}
0 \le \liminf_{n \to \infty} p_{\mu_n}(z_n) - p_{\nu_n}(z_n) = \liminf_{n \to \infty} - \frac{1}n \log |p_n(z_n)| + \frac{1}n \log |q_n(z_n)|.
\end{equation}
By canceling the common roots of $p_n$ and $q_n$, the right hand side of \eqref{E:pot-close} is equal to 
\begin{equation}
\label{E:zn-yni}
 \liminf_{n \to \infty} \frac{1}n \log \prod_{i=k(n) +1}^{n} \frac{1}{|z_n - y_{n, i}|}.
\end{equation}
Since $|y_{n, i}| \ge 2r$ for all $n, i$ and $|z_n| \le r$, there exist $0 < c_1 < c_2$ such that $|z_n - y_{n, i}|/|y_{n, i}| \in [c_1, c_2]$ for all $n, i$. Combining this with the fact that $k(n) = o(n)$, we get that the above liminf is equal to the same liminf with $|z_n - y_{n, i}|$ replaced by $|y_{n, i}|$. Hence by Lemma \ref{L:tightness} (ii), \eqref{E:zn-yni} is bounded below by $0$. 
\end{proof}

To prove Theorem \ref{T:det-cvg}, we will also need one more theorem about asymptotic magnitudes of sequences of polynomials. The proof of this theorem is rather long, and is very similar to parts of the proof of Theorem 1.1.4/Theorem 3.1.1 from \cite{stahl1992general}. As a result, we leave it to the appendix.

\medskip

\begin{theorem}
\label{T:stahl-mod-2}
Let $\{p_n\}$ be a sequence of degree-$n$ polynomials, and let $K \sset \C$ be a non-polar compact set. If assumption (i) of Theorem \ref{T:det-cvg} holds, then for any subsequential limit $\mu$ of $\mu_n$, we have that
$$
p_\mu(z) \ge p_{\mu_K}(z)
$$
for all $z \in \C$, with equality for all $z \in \close{U(K)}$.
\end{theorem}

We can now easily complete the proof of Theorem \ref{T:det-cvg}.

\begin{proof}[Proof of Theorem \ref{T:det-cvg}]
Again, we can assume that $a_{n, n} = 1$ for all $n$. By Lemma \ref{L:tightness}, the sequence $\mu_n$ is tight. Let $\mu$ be a subsequential limit of $\mu_n$ along a subsequence $Y \sset \N$. We have that
$$
p_\mu(z) \le \liminf_{n \to \infty} -\frac{1}{n} \log |p_n(z)| \le -\log (\capa(K)) = p_{\mu_K}(z)
$$
for almost every $z \in \Int(P(K))$. Here the first inequality comes from Lemma \ref{L:p-descent}, the second inequality comes from assumption (ii) of Theorem \ref{T:det-cvg}, and the final equality is by Theorem \ref{T:eq-def}. Combining this with Theorem \ref{T:stahl-mod-2} implies that $p_\mu = p_{\mu_K}$, and hence $\mu = \mu_K$.
\end{proof}

\subsection{The general case of Theorem \ref{T:det-cvg-in}}
\label{SS:in-all}

We now use Theorem \ref{T:det-cvg} to prove the general case of Theorem \ref{T:det-cvg-in}. We start with a proposition that will allow us to exchange a polynomial sequence $p_n$ satisfying the hypotheses of Theorem \ref{T:det-cvg-in} with a sequence that satisfies the hypotheses of Theorem \ref{T:det-cvg-in} with $i_n = 0$, at the expense of moving $o(n/\log n)$ roots.

\begin{prop}
\label{P:root-exchange}
Let $\{p_n(z) = \sum_{k=0}^n a_{n, k} z^k : n \in \N\}$ be a sequence of polynomials satisfying
 conditions (i) and (ii) of Theorem \ref{T:det-cvg-in} for a sequence $i_n = o(n/\log n)$. Then there exists a sequence of polynomials $\{q_n(z) = \sum_{k=0}^n b_{n, k} z^k : n \in \N\}$ satisfying hypotheses (i) and (ii) of Theorem \ref{T:det-cvg} such that $q_n$ and $p_n$ share all but at most $i_n$ roots.
\end{prop}

\begin{proof}
Let $\{z_{n, i} : i \in \{1, \dots, n\} \}$ be the roots of $p_n$, ordered so that $|z_{n, 1}| \ge |z_{n, 2}| \ge \dots \ge |z_{n, n}|$. Let $r > 0$ be large enough so that $K \sset D_r = \{z : |z| \le r\}$, and for each $n$ let $t_n$ be the largest natural number such that $|z_{n, t_n}| \ge 2r$. In the case when $|z_{n, 1}| < 2r$, we set $t_n = 0$. Set $s_n = \min(t_n, i_n)$, and define
$$
q_n(z) = b_{n, n} (z - 2r)^{s_n} \prod_{j=s_n +1}^n (z - z_{n, j}),
$$
where the coefficient $b_{n, n}$ is a positive real number, chosen so that $||q_n||_{L^p(\tau)} = ||p_n||_{L^p(\tau)}$. Now observe that for any $z \in K$, we have that
\begin{equation}
\label{E:pn-qn}
\frac{|q_n(z)|}{|p_n(z)|} = \frac{|b_{n, n}||z - 2r|^{s_n}}{|a_{n, n}||z - z_{n, 1}| \dots |z - z_{n, s_n}|} \in \lf[\frac{|b_{n, n}| }{|a_{n, n}||z_{n, 1}| \dots |z_{n, s_n}|}(r/2)^{s_n},  \frac{|b_{n, n}| }{|a_{n, n}||z_{n, 1}| \dots |z_{n, s_n}|}(6r)^{s_n}\rg].
\end{equation}
Therefore the condition that $||q_n||_{L^p(\tau)} = ||p_n||_{L^p(\tau)}$ and the fact that $s_n \le i_n = o(n/\log n)$ implies that
\begin{equation}
\label{E:sn-bd}
\lim_{n \to \infty} \frac{1}n \log \lf(\frac{|b_{n, n}|}{|a_{n, n}||z_{n, 1}| \dots |z_{n, s_n}|}\rg) = 0.
\end{equation}
Moreover, if $s_n < i_n$, then for each of the terms $z_{n, j}$ for $j \in \{s_n +1, \dots,  i_n\}$ we have $|z_{n, j}| \le 2r$. In particular, this implies that the equality in \eqref{E:sn-bd} gives
\begin{equation}
\label{E:in-bd}
\liminf_{n \to \infty} \frac{1}n \log \lf(\frac{|b_{n, n}|}{|a_{n, n}||z_{n, 1}| \dots |z_{n, i_n}|}\rg) \ge 0.
\end{equation}
Now, since the roots of $p_n$ are labelled in decreasing order of magnitude, Vieta's formula implies the bound
\begin{equation*}
|a_{n, n - i_n}| \le {n \choose i_n} |a_{n, n}| |z_{n, 1}||z_{n, 2}| \dots |z_{n, i_n}| .
\end{equation*}
Since $i_n = o(n/\log n)$, we have that 
$$
\log {n \choose i_n} \le i_n \log n = o(n).
$$
Therefore combining the above inequality for $|a_{n, n - i_n}|$ with \eqref{E:in-bd} gives that
\begin{equation}
\label{E:bn-ann}
\liminf_{n \to \infty} \frac{1}n \log \lf(\frac{|b_{n, n}|}{|a_{n, n - i_n}|}\rg) \ge 0.
\end{equation}
Since $q_n$ was chosen to have that same $L^p(\tau)$-norm as $p_n$, condition (i) of Theorem \ref{T:det-cvg-in} for $p_n$  then implies that 
\begin{equation}
\label{E:qn-log}
\limsup_{n \to \infty} \frac{1}{n} \log \lf(\frac{||q_n||_{L^p(\tau)}}{|b_{n, n}|} \rg) \le \limsup_{n \to \infty} \frac{1}{n} \log \lf(\frac{||p_n||_{L^p(\tau)}}{|a_{n, n - i_n}|} \rg)\le \log \capa(K).
\end{equation}
The regularity of $\tau$ implies that the inequalities above must in fact be equalities. This shows that the sequence $\{q_n\}$ satisfies condition (i) of Theorem \ref{T:det-cvg}. 

\medskip

Now, since each of the inequalities in \eqref{E:qn-log} is an equality, the inequality \eqref{E:bn-ann} must also be an equality, with a limit in place of a liminf. Combining this with the bound \eqref{E:pn-qn}, which also holds for all points in $P(K)$, we have that
\begin{equation}
\lim_{n \to \infty} \frac{1}{n} \log \lf( \frac{|p_n(z)||b_{n, n}|}{|a_{n, n - i_n}||q_n(z)|} \rg) = 0 
\end{equation}
for all $z \in P(K)$. This implies condition (ii) of Theorem \ref{T:det-cvg} for $\{q_n\}$.
 \end{proof}
 
 \begin{proof}[Proof of Theorem \ref{T:det-cvg-in}.]
By Proposition \ref{P:root-exchange}, we can relate the sequence $\{p_n\}$ to a sequence $\{q_n\}$ satisfying the conditions of Theorem \ref{T:det-cvg} such that all but at most $i_n$ roots of $q_n$ agree with those of $p_n$. The zero measures of $q_n$ converge weakly to $\mu_K$ by Theorem \ref{T:det-cvg}. Since $i_n = o(n)$, the distance in the weak topology between the zero measures of the sequences $q_n$ and $p_n$ goes to $0$ as $n \to \infty$, so the zero measures of $p_n$ also converge weakly to $\mu_K$.
 \end{proof}

Note that in the proof of Proposition \ref{P:root-exchange},  the second inequality in \eqref{E:qn-log} must in fact be an equality. This follows from \eqref{E:lower-lower}. This leads to a bound on non-leading coefficients of sequences of polynomials that is analogous to the known bound \eqref{E:lower-lower} on the leading coefficients. We state this here as a separate theorem as it may be of independent interest. 

\begin{theorem}
\label{T:non-leaders}
Let $K \sset \C$ be a compact non-polar set and let $\tau$ be a regular measure on $K$. Then for any $p \in (0, \infty]$, any sequence of polynomials $\{p_n(z) = \sum_{k=0}^n a_{n, k} z^k : n \in \N\}$ and any sequence of natural numbers $i_n = o(n/\log n)$, we have that
$$
\liminf_{n \to \infty} \frac{1}n \log \lf(\frac{||p_n||_{L^p(\tau)}}{|a_{n, n - i_n}|}\rg) \ge \log \capa(K).
$$
\end{theorem}

In a private communication, Vilmos Totik provided a proof of Theorem \ref{T:non-leaders} prior to us formulating and proving Proposition \ref{P:root-exchange}. Totik also observed that the theorem is no longer true when $i_n \ne o(n/\log n)$. For example, consider the usual Chebyshev polynomials $p(x) = 2\cos(n \arccos(x/2))$ on the compact set $[-2, 2]$.

\medskip

Before moving on, we prove one more necessary deterministic result.

\begin{lemma}
	\label{L:lemma-26}
	Let $\{p_n\}$ be a sequence of asymptotically minimal polynomials on a compact set $K \sset \C$. Then
	\begin{equation*}
	\limsup_{n \to \infty} \frac{1}n \log |p_n(z)| \le 0 \qquad \text{uniformly on compact subsets of $\Int(P(K))$}.
	\end{equation*}
\end{lemma}

\begin{proof}
First, by Lemma \ref{L:tightness} and the asymptotic minimality of $\{p_n\}$, the sequence of zero measures $\{\mu_n\}$ is tight (note that asymptotic minimality implies assumption (i) of Theorem \ref{T:det-cvg}). Letting $\mu$ be any subsequential limit of $\{\mu_n\}$, Theorem \ref{T:stahl-mod-2} implies that
$$
p_\mu(z) \ge -\log \capa (K)
$$
for all $z \in \Int P(K)$. Therefore by Lemma \ref{L:p-descent}, we have that
$$
\limsup_{n \to \infty} \frac{1}n \log |p_n(z)| - \frac{1}n \log |a_{n, n}| \le \log \capa (K)
$$
uniformly on $\Int(P(K))$. The second term on the left hand side above converges to $\log \capa(K)$ by asymptotic minimality, proving the lemma.
\end{proof}
\section{Sufficiency}
\label{S:sufficiency}

In this section we prove the sufficiency statements in Theorem \ref{T:extremal} by checking that the conditions of Theorem \ref{T:det-cvg-in} hold (either almost surely or in probability). We first check assumption (i). We start with two basic lemmas about sequences of random variables.

\begin{lemma}
\label{L:basic-xi}
Let $\{\xi_i : i \in \N\}$ be a sequence of i.i.d.\ complex non-zero random variables, and let
$$
L_n= \max_{i \in \{1, \dots, n\}} \frac{1}n \log |\xi_i|.
$$
\begin{enumerate}[nosep, label=(\roman*)]
\item If $\expt \log(1 + |\xi_0|) < \infty$, then $L_n \to 0$ almost surely.
\item If $\p (|\xi_0| > e^n) = o(n^{-1})$, then $L_n \to 0$ in probability.
\end{enumerate}
\end{lemma}

\begin{proof}
For (i) for $L_n$, the condition on the random variables implies that for every $\ep > 0$,
$$
\sum_{n = 0}^\infty \prob( \log(1 + |\xi_n|) > \ep n) < \infty.
$$
By the Borel-Cantelli Lemma, this implies that for every $\ep > 0$, that
\begin{equation*}
\prob \lf(\frac{1}n \log |\xi_n| > \ep \text{ infinitely often} \rg) = 0.
\end{equation*}
This immediately implies that $L_n \to 0$ almost surely. For (ii), the condition on the random variables and a union bound implies that for $\ep > 0$, we have
$$
\lim_{n \to \infty} \prob( \exists i \in \{0, \dots, n \} \text{ such that } \log(1 + |\xi_i|) > \ep n) \to 0.
$$
This implies that $L_n \to 0$ in probability. 
\end{proof}

\begin{lemma}
\label{L:zeta-bd}
Let $\xi_i$ be a sequence of complex i.i.d.\ non-degenerate random variables, and let $\{a_{n, k} : k \le n \in \{0, 1, \dots\}\}$ be a deterministic triangular array of complex numbers satisfying
\begin{equation}
\label{E:an-cond}
\lim_{n \to \infty} \frac{1}n \log|a_{n, n}| = c
\end{equation}
for some $c \in \R$. Define $\zeta_{n, j} = \sum_{k = j}^n \xi_k a_{k, j}$, and let 
$$
n - I_n = \argmax_{j \in \{\floor{n - \log^2 n}, \dots, n\}} |\zeta_{n, j}|.
$$
If more than one value of $|\zeta_{n, j}|$ obtains the maximum above set $n - I_n$ to be the largest such value of $j$. Then
$$
\liminf_{n \to \infty} \frac{1}n \log |\zeta_{n, n - I_n}| \ge c.
$$
\end{lemma}

\begin{proof}
Let $\scrF_{j, n}$ denote the $\sig$-algebra generated by the random variables $\xi_{j}, \dots, \xi_n$, and let $I_{j, n}$ denote the event where $|\zeta_{n, i}| \le m$ for all $i \in \{j, \dots, n\}$. Also letting $I_{n+1, n}$ denote the whole probability space, we have the telescoping product
\begin{equation}
\label{E:zeta-split}
\p(|\zeta_{n, n - I_n}| \le m) = \prod_{j =  \floor{n - \log^2 n}}^n \frac{\p \big(|\zeta_{n, j}| \le m , I_{j+1, n} \big)}{\p I_{j+1, n}}. 
\end{equation}
Now,
\begin{align*}
\p \big(|\zeta_{n, j}| \le m , I_{j+1, n} \big) &= \E\p( |\zeta_{n, j}| \le m, I_{j+1, n} \;|\; \scrF_{j+1, n}) \\
&= \E\mathbf{1}(I_{j+1, n}) \p( |\zeta_{n, j}| \le m \;|\; \scrF_{j+1, n}) \\
&\le M_j \p I_{j+1, n},
\end{align*}
where $M_j$ is the maximum value of the conditional probability $\p( |\zeta_{n, j}| \le m \;|\; \scrF_{j+1, n})$. Here the second equality follows since $I_{j+1, n}$ is $\scrF_{j+1, n}$-measurable. Therefore the right side of \eqref{E:zeta-split} is bounded above by 
$$
\prod_{j=\floor{n - \log^2 n}}^n M_j.
$$
Now,
\begin{align*}
M_j = \max \p\lf(\lf|\xi_j a_{j, j} - Z \rg| \le m \; \big| \; \scrF_{j+1, n} \rg),
\end{align*}
for an $\scrF_{j+1, n}$-measurable random variable $Z \in \C$.
Since the $\{\xi_i: i \in \N\}$ are independent, $\xi_j$ is independent of $\scrF_{j+1, n}$. Therefore
\begin{align*}
M_j \le \max_{z \in \C} \p \lf( |\xi_j a_{j, j} - z| \le m \rg) = \max_{z \in \C} \p \lf( |\xi_j - z| \le m |a_{j, j}|^{-1} \rg).
\end{align*}
Since the $\xi_i$ are non-degenerate and i.i.d.\, there exists a $\ga > 0$ such that whenever $m |a_{j, j}|^{-1} \le \ga$, the right hand side above is less than or equal to $1- \de$ for some $\de >0$. Hence for any $\ep > 0$, the condition \eqref{E:an-cond} on the coefficients $a_{j,j}$ implies that for all large enough $n$ we have
$$
\p(|\zeta_{n, n - I_n} | \le e^{n(c - \ep)}) \le \prod_{j =  \floor{n - \log^2 n}}^n M_j \le \prod_{j =  \floor{n - \log^2 n}}^n  \max_{z \in \C} \p \lf( |\xi_j - z| \le e^{- \ep n/2} \rg) \le (1 - \de)^{\log^2 n}.
$$
The right hand side above is summable, and hence by the Borel-Cantelli lemma, we have that
$$
\p(|\zeta_{n, n - I_n}| \le e^{n(c - \ep)} \text{ infinitely often}) = 0. 
$$
Taking logarithms and dividing by $n$ then proves the lemma.
\end{proof}

We can now combine Lemma \ref{L:basic-xi} with Lemma \ref{L:zeta-bd} to show that condition (i) of Theorem \ref{T:det-cvg-in} holds for the random asymptotically minimal polynomials that we are working with.

\begin{lemma}
\label{L:basic-bd}
Let $\{p_n = \sum_{i=0}^n a_{n, i} z^i \}$ be a sequence of asymptotically minimal polynomials on a compact set $K$, and let
$$
G_n = \sum_{i=0}^n \xi_i p_i = \sum_{i=0}^n \zeta_{n, i} z^i
$$
be the random polynomials formed from the sequence $\{p_n\}$ for a sequence of i.i.d.\ non-degenerate complex random variables $\xi_i$.
\begin{enumerate}[label=(\roman*)]
\item If $\expt \log(1 + |\xi_0|) < \infty$, then there exists a random sequence of natural numbers $I_n$ such that $I_n = o(n/ \log n)$ almost surely, and such that for almost every $\om \in \Om$,
\begin{equation}
\label{E:asym-Gn}
\lim_{n \to \infty} \frac{1}n \log \lf(||G_n(\cdot, \om)||_{L^p(\tau)} \rg) = 0 \quad \mathand \quad \lim_{n \to \infty} \frac{1}n \log |\zeta_{n, n - I_n}| = - \log \capa(K).
\end{equation}
Here $\tau, p$ are as in Definition \ref{D:asym-ext} for the sequence $\{p_n\}$.
\item If $\p \lf( |\xi_0| > e^n\rg) = o(n^{-1})$, then there exists a random sequence of natural numbers $I_n$ such that $I_n = o(n/ \log n)$ almost surely, and such that both equations in \eqref{E:asym-Gn} hold in probability.
\end{enumerate}
\end{lemma}

\begin{proof} We will prove (i) and (ii) together. First, by the definition of asymptotically minimal polynomials (Definition \ref{D:asym-ext}), the random variables $\xi_i$ and the coefficient array $\{a_{n, k}\}$ satisfy the hypotheses of Lemma \ref{L:zeta-bd} with $c = - \log \capa K$, and so there exists a sequence $I_n = o(n/\log n)$ almost surely such that
\begin{equation}
\label{E:logcap-sup}
\liminf_{n \to \infty} \frac{1}n \log \ |\zeta_{n, n - I_n}| \ge - \log \capa K \qquad \as.
\end{equation}
Next, we restrict our attention to $p \in (0, 1]$. Letting $L_n$ be as in Lemma \ref{L:basic-xi}, we have that 
\begin{align*}
\frac{1}n \log ||G_n||_{L^p(\tau)} \le \frac{1}{pn} \log \lf(\sum_{i=1}^n |\xi_i|^p||p_i||^p_{L^p(\tau)} \rg) \le L_n + \frac{1}{pn} \log \lf(\sum_{i=1}^n ||p_i||^p_{L^p(\tau)}\rg).
\end{align*}
By Lemma \ref{L:basic-xi}, the first term on the right hand side above converges to $0$ as $n \to \infty$, either in probability or almost surely, depending on our assumptions on the random variables. The second term on the right hand side converges to zero by the asymptotic minimality of $\{p_i\}$, giving that
\begin{equation}
\label{E:Gn-up}
\limsup_{n \to \infty} \frac{1}n \log ||G_n||_{L^p(\tau)} \le 0,
\end{equation}
either almost surely or in probability depending on our underlying assumptions.
The same computation works for $p \in (1, \infty]$, except that we do not need to raise the $L^p$-norm $||G_n||^p_{L^p(\tau)}$ to the $p$th power before applying the triangle inequality.

\medskip

Combining the bounds in \eqref{E:logcap-sup} and \eqref{E:Gn-up} with the lower bound in Theorem \ref{T:non-leaders} implies that both \eqref{E:logcap-sup} and \eqref{E:Gn-up} must in fact be equalities (either almost surely or probability) with limits in place of the liminf and limsup.
\end{proof}

To establish condition (ii) of Theorem \ref{T:det-cvg-in} on the polynomials $G_n$, we use a result from \cite{bloom2018universality}. This lemma is a corollary of a small ball probability result of Nguyen and Vu \cite{nguyen2011optimal} adapted to proving convergence of logarithmic potentials. 

\begin{lemma}
\label{L:as-criterion} [Lemma 6.2, \cite{bloom2018universality}]
 Let $\{a_{n} : n \in \N \}$ be a sequence of complex numbers such that
\begin{equation}
\label{E:A-limit}
\lim_{n \to \infty} \frac{1}{n} \log \left( \sum_{i=0}^{n} |a_{i}| \right) = a.
\end{equation}
Let $||a^{(n)}||$ be the Euclidean norm of $(a_{0}, \dots , a_{n})$, and let $w_{n, i} = a_{i}/||a^{(n)}||$. Suppose that for any $\epsilon > 0$, there exists a $\delta > 0$ such that for all large enough $n$, the set 
$$
\mathcal{W}_n = \{w_{n, i} : 0 \le i \le n \}
$$
cannot be covered by a union of $n^{2/3 + \delta}$ balls of radius $e^{-\epsilon n}$. If $\{\xi_0, \xi_1, \dots \}$ is a sequence of non-degenerate i.i.d.\ complex random variables, then
$$
\liminf_{n \to \infty} \frac{1}n \log \left| \sum_{i=0}^n \xi_i a_{i} \right| \ge a \quad \text{ almost surely}.
$$

\end{lemma}

By Lemma \ref{L:as-criterion}, if we can show that for almost every $z \in \Int(P(K))$, the values $\{p_i(z) : i \in \N\}$ are sufficiently well spaced out, then we can prove assumption (ii) of Theorem \ref{T:det-cvg-in}. To do this, we require Cartan's estimate on the measure of the set where a polynomial can take on small values (see \cite{levin1996lectures}, Lecture 11). Here and throughout the remainder of the paper, $\scrM$ is planar Lebesgue measure on $\C$.

\begin{lemma}[Cartan's estimate]
\label{L:poly-estimate}
Let $p$ be a degree $n$ monic polynomial. Then for any $h > 0$,
$$
\scrM \{ z: |p(z)| \le h^n \} \le 25\pi e^2 h^2.
$$
\end{lemma}

We can now prove the following preliminary version of assumption (ii) for Theorem \ref{T:det-cvg-in}. 

\begin{prop}
\label{P:on-int}
Let $\{p_n\}$ be a sequence of asymptotically minimal polynomials, and let $\{\xi_i\}$ be a sequence of non-degenerate i.i.d.\ complex random variables. Set $G_n = \sum_{i=0}^n \xi_i p_i$. For almost every $z \in \Int(P(K))$, we have that
\begin{equation}
\label{E:as-want-orthog}
\liminf_{n \to \infty} \frac{1}n \log |G_n(z, \om)| \ge 0\qquad \text{ for almost every } \om \in \Om.
\end{equation}
\end{prop}

\begin{proof}
We first show that 
\begin{equation}
\label{E:lim-00}
\lim_{n \to \infty} \frac{1}n \log \lf(\sum_{i=0}^n |p_i(z)| \rg) = 0 
\end{equation}
for every $z \in \Int(P(K))$ outside of the finite set $F = \{z \in \Int(P(K)) : p_i(z) = 0 \mathforall i \in \N\}$. Indeed, since $\sum_{i=0}^n |p_i(z)|$ is nondecreasing in $n$, the $\liminf$ of the right hand side of \eqref{E:lim-00} is nonnegative as long as at least one value of $p_i(z)$ is nonzero. On the other hand, we have
$$
\frac{1}n \log \lf(\sum_{i=0}^n |p_i(z)| \rg) \le \frac{1}n \log \lf((n+1) \max_{i \le n} |p_i(z)| \rg) = \frac{\log (n + 1)}{n} + \frac{1}n \max_{i \le n} \log |p_i(z)|.
$$
The first term on the right hand side above converges to $0$, and the $\limsup$ of the second term is $0$ by Lemma \ref{L:lemma-26}, yielding \eqref{E:lim-00}.

\medskip
Now let $w_n(z) = p_n(z)/||p^{(n)}(z)||$, where $||p^{(n)}(z)||$ is the Euclidean norm of $(p_0(z), p_1(z), \ddd, p_n(z))$. For $z \in \C, n \in \N$, define
$$
\scrW^n_z = \lf \{w_0(z), \dots, w_n(z) \rg\}.
$$
Let $V$ be a compact subset of $\Int(P(K))$. Define
$$
A_{\al, n} = \{z \in V : \scrW^n_z \text{ can be covered by a union of $\floor{n^{3/4}}$ balls of radius $e^{-\al n}$}\}.
$$
We will show that for any $\al > 0$, the set
\begin{align*}
B_{\al} = \Big \{z \in V : z \in A_{\al, n} \quad \text{for infinitely many $n$}\Big\}
\end{align*} 
has Lebesgue measure $0$. Once we have this, Lemma \ref{L:as-criterion} and \eqref{E:lim-00} implies that for every
$$
z \in V \smin \lf( \bigcup_{m \in \N} B_{1/m} \cup F \rg),
$$
the convergence in \eqref{E:as-want-orthog} holds. This set has full Lebesgue measure in $V$. Since $V$ was chosen arbitrarily, this will imply that \eqref{E:as-want-orthog} holds for almost every $z \in \Int(P(K)).$

\medskip

If $z \in \Int(P(K))$ is such that $\scrW^n_z$ can be covered by a union of $\floor{n^{3/4}}$ balls of radius $e^{-\al n}$, then there must be two points $w_{m_1}(z), w_{m_2}(z) \in \scrW_z^n$ with 
$$
n^{3/4} \le m_1 < m_2 \le 2n^{3/4} \qquad \mathand \qquad |w_{m_1}(z) - w_{m_2}(z)| < 2e^{-\al n}.
$$
In other words, $|p_{m_2}(z) - p_{m_1}(z)| < 2||p^{(n)}(z)||e^{-\al n}$. Now by Lemma \ref{L:lemma-26}, for all large enough $n$,
$$
2||p^{(n)}(z)||e^{-\al n} < e^{-\al n/2} \qquad \text{ for all } z \in V.
$$
For such $n$, the set $A_{\al, n}$ is contained in
$$
A_{\al, n}^* = \{z \in V : \text{ there exists $n^{3/4} \le m_1 < m_2 \le 2n^{3/4}$  such that $|p_{m_1}(z) - p_{m_2}(z)| < e^{-\al n/2}$ } \}.
$$
By Cartan's estimate (Lemma \ref{L:poly-estimate}), for any $n^{3/4} \le m_1 < m_2 \le 2n^{3/4}$, we have that 
\begin{equation}
\label{E:cartan-inside}
\scrM\{  z : |p_{m_2}(z) - p_{m_1}(z)| < e^{- \al n/2} \} \le 25 \pi \exp \lf(2 - \frac{\al n}{m_2} - \frac{2}{m_2} \log |a_{m_2, m_2}| \rg).
\end{equation}
The sequence $\{\frac{1}j \log |a_{j, j}|\}$ has a limit by the asymptotic minimality of $\{p_n\}$, and hence is uniformly bounded for large enough $j \in \N$ . Therefore the right hand side of \eqref{E:cartan-inside} is bounded above by $k\exp(-\al n^{1/4})$ 
for some constant $k$ independent of $n, m_2$, and $m_1$. Hence for all large enough $n$, a union bound gives that
$$
\scrM(A_{\al, n}) \le \scrM(A_{\al, n}^*) \le k n^{3/2} \exp(-\al n^{1/4}),
$$
so $\sum_{n=0}^\infty \scrM(A_{\al, n}) < \infty$. By the Borel-Cantelli Lemma, $B_\al$ has Lebesgue measure $0$.
\end{proof}

We can now prove the sufficiency statements in Theorem \ref{T:extremal}.

\begin{proof}
We first prove almost sure convergence under the condition $\expt \log(1 + |\xi_0|) < \infty$. By Lemma \ref{L:basic-bd}, condition (i) of Theorem \ref{T:det-cvg-in} is satisfied for almost every $\om \in \Om$ for the (random) sequence $I_n$ identified in that lemma. Now, let $J(z, \om)$ denote the indicator of event where
$$
\liminf_{n \to \infty} \frac{1}n \log \lf(|G_n(z, \om)| \rg) < 0.
$$
We think of $J$ as a function from the space $\Int(P(K)) \X \Om$ equipped with the product of Lebesgue measure on $\Int(P(K))$ and the natural probability measure $\p$ on $\Om$. Rephrased in this language, Proposition \ref{P:on-int} gives that
$$
\int_{\Int(P(K))} \int_\Om J(z, w) d\p (w) d \scrM(z) = 0.
$$
By Fubini's theorem, this equality is also true with the order of integration reversed and so $\int J(z, w) d \scrM(z) = 0$ for almost every $\om \in \Om$. Therefore for almost every $\om \in \Om$, we have $J(z, w) = 0$ for almost every $z \in \Int(P(K))$, or equivalently,
$$
\liminf_{n \to \infty} \frac{1}n \log \lf(|G_n(z, \om)| \rg) \ge 0 \qquad \text{ for almost every } z \in \Int(P(K)).
$$
The above bound combined with Lemma \ref{L:basic-bd}(i) gives that
$$
\liminf_{n \to \infty} \frac{1}n \log \lf(|\zeta_{n, n - I_n}|^{-1} |G_n(z, \om)| \rg) \ge \log \capa(K) \qquad \text{ for almost every } z \in \Int(P(K)).
$$
Hence condition (ii) of Theorem \ref{T:det-cvg-in} is also satisfied for almost every $\om \in \Om$ with the same sequence $I_n$, and so $\mu_{G_n} \to \mu_K$ almost surely.

\medskip

We now show convergence in probability under the condition $\p(|\xi_0| > e^n) > o(n^{-1})$. Let $I_n$ be the sequence identified in Lemma \ref{L:basic-bd}. By Lemma \ref{L:basic-bd} (ii), for any subsequence $Y \sset \N$ we can find a further subsequence $Y_0 \sset Y$ such that condition (i) of Theorem \ref{T:det-cvg-in} holds almost surely when the corresponding limit is taken over $Y_0$ with the above choice of $I_n$. Also, by the same reasoning as in the almost sure case, condition (ii) of Theorem \ref{T:det-cvg-in} holds almost surely along $Y_0$ with the same choice of $I_n$. Therefore $\{\mu_n: n \in Y_0\}$ converges to $\mu_K$ almost surely. Since $Y$ was arbitrary, $\mu_n \to \mu_K$ in probability.
\end{proof}

\section{Necessity}
\label{S:necessity}
In this section, we prove the necessity statements in Theorem \ref{T:extremal}, completing the proof of that theorem. As discussed in Section \ref{SS:main-results}, we will do this by showing that one term $\xi_ip_i$ dominates the other terms in $G_n$ and applying Rouch\'e's theorem. To do so, we require a few lemmas about the magnitude of i.i.d.\ random variables that fail the moment conditions of Theorem \ref{T:extremal}, and one lemma about the magnitude of polynomials on an annulus.

\medskip

We start with a lemma that will be used to show the necessity of the condition $\expt \log(1 + |\xi_0|) < \infty$ for almost sure convergence.

\begin{lemma}
\label{L:cond-as}
Suppose that $\{\xi_i : i \in \N\}$ is a sequence of i.i.d.\ random variables such that 
$$
\expt \log(1 + |\xi_0|) = \infty \qquad \mathand \qquad \prob(|\xi_0| > e^n) = o(n^{-1}).
$$
Let $A_{n, c}$ be the event where
$$
|\xi_{\floor{n/2}}| \ge e^{(c+1)n}, \mathand \; |\xi_j| < e^{n} \text{ for all } \;\; j \in [0, n], j \ne \floor{n/2}.
$$
For every fixed $c \in \real$, infinitely many of the events $\{A_{n, c} : n \in \N\}$ occur almost surely.
\end{lemma}

To prove this we will use the following strengthening of the Borel-Cantelli Lemma, due to Kochen and Stone (see \cite{fristedt2013modern}, Chapter 6.2).

\begin{lemma}
\label{L:borel-cant}
Let $\{B_i: i \in \N\}$ be a sequence of events such that
$$
\sum_{i=1}^\infty \prob(B_i) = \infty \qquad \mathand \qquad \limsup_{n \to \infty} \frac{\sum_{i, k = 1}^n \prob(B_i)\prob(B_k)}{\sum_{i, k = 1}^n \prob(B_i \cap B_k)} = L.
$$
Then $\prob(\text{Infinitely many} \; B_i \; \text{occur}) \ge L.$

\end{lemma}

\begin{proof}[Proof of Lemma \ref{L:cond-as}.]
First, let
$$
\beta_x = \prob(|\xi_0| \ge e^x).
$$
For each $n, c$, we have that
$$
\prob(A_{n, c}) = \beta_{(c+1)n}(1 - \beta_n)^{n}.
$$
By the two assumptions of the lemma, we have that
\begin{equation}
\label{E:two-cons}
\lim_{n \to \infty} (1 - \beta_n)^n = 1 \qquad \mathand \qquad \lim_{n \to \infty} \sum_{i=1}^n \beta_{(c+1)i} = \infty,
\end{equation}
and thus
$
\sum_{n=1}^\infty \p(A_{n, c}) = \infty.
$
Moreover, for $m, n \in \N$ with $\floor{m/2} \ne \floor{n/2}$, we have that
$$
\prob(A_{n, c} \cap A_{m, c}) \le \prob(|\xi_{\floor{n/2}}| \ge e^{(c+1)n}, \;\; |\xi_{\floor{m/2}}| \ge e^{(c+1)m}) = \beta_{(c+1)m}\beta_{(c+1)n}.
$$
Therefore
\begin{align*}
\sum_{i, k = 1}^n \p(A_{i, c} \cap A_{k, c}) &\le \sum_{|i - k| \ge 2}^n \p(A_{i, c} \cap A_{k, c}) + 3\sum_{i=1}^n \p(A_{i, c}) \\
&\le \sum_{i, k=1}^n \beta_{(c+1)i}\beta_{(c+1)k} + 3\sum_{i=1}^n \beta_{(c+1)i} = \sum_{i=1}^n \beta_{(c+1)i} \lf( \sum_{k=1}^n \beta_{(c+1)k} + 3 \rg)
\end{align*}
Hence we have that 
\begin{equation}
\label{E:fracas}
\frac{\sum_{i, k = 1}^n \prob(A_{i, c})\prob(A_{k, c})}{\sum_{i, k = 1}^n \prob(A_{i, c} \cap A_{k, c})} \ge 
\frac{\sum_{i = 1}^n \beta_{(c+1)i}(1 - \beta_i)^{i} \lf( \sum_{k=1}^n \beta_{(c+1)k} (1 - \beta_k)^{k} \rg)}{\sum_{i=1}^n \beta_{(c+1)i} \lf( \sum_{k=1}^n \beta_{(c+1)k} + 3 \rg).
}
\end{equation}
By the two facts in \eqref{E:two-cons}, the right hand side of \eqref{E:fracas} converges to $1$ as $n \to \infty$. By Lemma \ref{L:borel-cant}, this implies that infinitely many of the events $A_{n, c}$ occur almost surely.
\end{proof}

For showing the necessity of the condition $\prob(|\xi_0| > e^n) = o(n^{-1})$ for convergence in probability, we need the following lemma from \cite{bloom2018universality}.

\begin{lemma}[Lemma 5.5, \cite{bloom2018universality}]
\label{L:tail-bd} Let $X$ be a non-negative real random variable. Suppose that 
\begin{equation}
\label{E:ls-cond}
\limsup_{n \to \infty} n\mathbb{P}(X > n) > 0.
\end{equation}
Then there exists a function $f:[0, \infty) \to [0, \infty)$ such that 
\begin{enumerate}[label=(\roman*)]
\item The quantity
$$
C(f) =\limsup_{n \to \infty} n\mathbb{P}(X > f(n))
$$
is positive and finite.
\item For every $x, y \in [0, \infty)$, we have that
$
f(x) + y \le f(x + y).
$
\end{enumerate}

\end{lemma}

We will use the following corollary of Lemma \ref{L:tail-bd}. The proof of this corollary is quite similar to a statement shown in the proof of Theorem 5.6 in \cite{bloom2018universality}. We nonetheless include it for completeness.

\begin{corollary}
\label{C:in-prob-borel}
Let $\{\xi_i: i \in \N\}$ be a sequence of i.i.d.\ random variables with
$$
\limsup_{n \to \infty} n\prob(|\xi_0| > e^n) > 0.
$$
For $c \in \R$ and $n \in \N$, define
$$
B_{n, c} = \{\exists j \in [n/4, n/2] : |\xi_j| \ge e^{cn}|\xi_i| \mathforall i \in [0, n], i \ne j \}.
$$
Then for every $c \in \real$, we have that
$$
\limsup_{n \in \N} \prob(B_{n, c}) > 0.
$$
\end{corollary}

\begin{proof}
For a function $g: [0, \infty) \to [0, \infty)$ and $n \in \mathbb{N}$, define
$$
D_n(g) := n\mathbb{P}(|\xi_0| > g(n)), \quad \text{and} \quad D(g) := \limsup_{n \to \infty} D_n(g).
$$
We apply Lemma \ref{L:tail-bd} to the random variable $\log (1 + |\xi_0|)$ to obtain a function $f$ satisfying properties (i) and (ii) of that lemma. Letting $g = e^f$, we then have that
\begin{enumerate}[label=(\roman*)]
\item $D(g) \in (0, \infty)$.
\item For every $x, y \in [0, \infty)$, we have that 
$
g(x + y) \ge e^yg(x).
$
\end{enumerate}
For $\alpha \in (0, \infty)$, define $g_\alpha(x) : = g(\alpha x)$. Observe that $\alpha D(g_{\alpha}) = D(g)$.
Now define
$$
E_{n, \alpha} = \#\{ i \le n : |\xi_i| > g_\alpha(n) \} \qquad \mathand \qquad \tilde{E}_{n, \alpha} = \# \{ i \in [n/4, n/2] : |\xi_i| > g_\alpha(n) \}
$$
For each $\alpha$, $E_{n, \alpha}$ is a binomial random variable with $n$ trials and mean $D_n(g_\alpha)$. The random variable $\tilde{E}_{n, \alpha}$ is binomial with $m_n := \card{\Z \cap [n/4, n/2] } $ trials and mean $D_n(g_\alpha) m_n/n$. Of course, $m_n/n \to 1/4$ as $n \to \infty$.

\medskip

Now fix $c \in [0, \infty)$. For any $\alpha > c$, there exists a subsequence $Y \sset \N$ such that
$$
\lim_{n \in Y} \mathbb{E} \tilde{E}_{n, \alpha} = \frac{D(g)}{4\alpha}, \quad \text{whereas} \quad \limsup_{n \in Y} \mathbb{E} E_{n, \alpha - c} \le \frac{D(g)}{\alpha - c}.
$$
Therefore for large enough $\alpha$, property (i) and Poisson convergence for binomial random variables implies that
\begin{equation}
\label{E:abb}
\limsup_{n \in Y} \Big(\mathbb{P} (\tilde{E}_{n, \alpha} = 1) - \mathbb{P} (E_{n, \alpha - c} \ge 2) \Big) > 0.
\end{equation}
By property (ii) of the function $g$, we have that
$
B_{n, c} \sset \{ \tilde{E}_{n, \alpha} = 1, E_{n, \alpha - c} < 2 \},
$
and hence \eqref{E:abb} implies the lemma.
\end{proof}

The next lemma bounds the magnitude of a monic polynomial on an annulus.

\begin{lemma}
\label{L:lower-bd}
Let $0 < r_1 < r_2$. Then for any monic polynomial $q$ of degree $n \ge 1$, we can find a simple closed curve $C$ satisfying the following conditions:
\begin{enumerate}[nosep, label=(\roman*)]
\item $C$ is contained in the annulus $A_{r_1, r_2} = \{z : r_1 < |z| < r_2\}$.
\item The disk $D_{r_1} = \{z : |z| < r_1\}$ is contained in the interior of $C$.
\item $|q(z)| \ge [(r_2 - r_1)/5]^n$ for all $z \in C$.
\end{enumerate}
 
\end{lemma}

\begin{proof}
Fix a polynomial $q$ of degree $n$, and let $\mu$ be uniform measure on the roots of $q$. For $\al > 0$, let 
$$
S_\al = \lf\{z \in \C : \frac{1}{n} \log |q(z)| \le \log \al \rg\}.
$$
By Theorem 5.2.5 in \cite{ransford1995potential}, $\capa(S_\al) \le \al$. By a standard estimate on the diameter of a connected set in terms of its capacity (see \cite{ransford1995potential}, Theorem 5.3.2), the diameter of each of the connected components of $S_\al$ is at most $4\al$. Hence there is a simple closed curve $C$ contained in the annulus $A_{r_1, r_2}$ which avoids the set $S_{(r_2 - r_1)/5}$ and contains $D_{r_1}$ in its interior. The curve $C$ satisfies the conditions of the lemma.
\end{proof}

We are now ready to prove the necessity statements in Theorem \ref{T:extremal}. For this lemma, we use the notation $A_{n, c}$ and $B_{n, c}$ for the events in Lemma \ref{L:cond-as} and Corollary \ref{C:in-prob-borel}.

\begin{proof} 
\textbf{The necessity of the condition $\prob(|\xi_0| > e^n) = o(n^{-1})$ for convergence in probability:} 

Since $|p_n| = |a_{n, n}|e^{-np_{\mu_n}}$, we can bound the value of $|p_n|$ above by bounding $|a_{n, n}|$ above and $p_{\mu_n}$ below. By the asymptotic minimality of $p_n$, there exists $c > 0$ such that $|a_{n, n}| \le c^n$ for all $n$. By the asymptotic minimality of $p_n$, Lemma \ref{L:p-descent}, and Theorem \ref{T:stahl-mod-2}, for any compact set $K$, the logarithmic potentials $\{p_{\mu_n} : n \in \N\}$ are uniformly bounded below on $K$.

\medskip

Hence for every $r > 0$ there exists a constant $b_r > 0$ such that
\begin{equation}
\label{E:local-bound}
|p_n(z)| \le {b_r^n}
\end{equation}
for all $z \in D_{r+1} = \{z : |z| < r + 1\}$ and all $n \in \N$. Also, by the asymptotic minimality of the polynomials $\{p_n\}$, we can find a constant $d > 1$ such that $|a_{n, n}| \ge d^{-n}$ for all large enough $n \in \N$. Hence by Lemma \ref{L:lower-bd}, for all large enough $j \in \N$, we can find a simple closed curve $C_j \sset \{z : |z| \in (r, r+ 1)\}$ containing the disk $D_r$,  such that
\begin{equation}
\label{E:p-bd}
|p_j(z)| \ge \lf(\frac{1}{5d}\rg)^j \qquad \text{ for all } z \in C_j.
\end{equation}
Combining \eqref{E:local-bound} and \eqref{E:p-bd}, we can choose $c > 0$ such that for all large enough $n \in \N$, for every $j \in [n/4, n/2]$ we have that
\begin{equation}
\label{E:ecj}
e^{cn} |p_j(z)| > \sum_{i \in [0, n]} \lf|p_i(z) \rg| \qquad \text{ for all } z \in C_j.
\end{equation}
In particular, this implies that for large enough $n$, on the event $B_{n, c}$ there exists a (random) $J \in [n/4, n/2]$ such that
\begin{equation}
\label{E:xi-j-lower}
|\xi_J p_J(z)| >  \sum_{i \in [0, n], i \ne j} |\xi_i| |p_i(z)| \ge\lf| \sum_{i \in [0, n], i \ne j} \xi_i p_i(z) \rg| \qquad \text{ for all } z \in C_J.
\end{equation}
Hence by Rouch\'e's Theorem, $\mu_{G_n}(D_r) \le 1/2$ on the event $B_{n, c}$.  By Corollary \ref{C:in-prob-borel}, this implies that
$$
\limsup_{n \to \infty} \prob(\mu_{G_n}(D_r) \le 1/2) > 0.
$$
Since $r > 0$ was arbitrary, the sequence $\mu_{G_n}$ cannot have a limit in probability.

\medskip

\textbf{The necessity of the condition $\expt \log(1 + |\xi_0|) < \infty$ for almost sure convergence:} 

\medskip

To prove the necessity of $\expt \log(1 + |\xi_0|) < \infty$ for almost sure convergence, we may assume that the random variables $\xi_i$ satisfy the assumptions of Lemma \ref{L:cond-as}, since otherwise the sequence $\{\mu_{G_n}\}$ does not even converge in probability.

\medskip

In this case, for almost every $\om \in \Om$, for every $c > 0$ infinitely many of the events $A_{n, c}$ occur almost surely, where the events $A_{n, c}$ are as in Lemma \ref{L:cond-as}. Moreover, $A_{n, c} \sset B_{n, c}$ for all $n$. Hence by the above argument, the sequence $ \mu_{G_n}$ has no almost sure limit.
\end{proof}

\section{Appendix: Proof of Theorem \ref{T:stahl-mod-2}}
\label{S:appendix}

For this proof, we introduce the \textbf{fine topology} on $\C$. This is the coarsest topology on $\C$ that makes every subharmonic function continuous. In particular, we will use that if $A$ is a connected open set in the usual Euclidean topology, then the fine boundary of $A$ and the Euclidean boundary of $A$ coincide (see \cite{stahl1992general}, Appendix II).

\begin{proof}[Proof of Theorem \ref{T:stahl-mod-2}] 
Without loss of generality, we may assume that $a_{n, n} = 1$. We also let $p \in (0, \infty]$ and $\tau$ be a regular measure on $K$ such that assumption (i) of Theorem \ref{T:det-cvg} holds.

\medskip

The sequence $\{\mu_n\}$ is tight by Lemma \ref{L:tightness} with subsequential limits supported on $P(K)$. Let $\mu_{n_i}$ be a convergent subsequence of $\mu_n$ with limit $\mu$. For ease of notation, we relabel $\mu_{n_i} = \mu_n$. 
We will show that 
\begin{equation}
\label{E:on-bry}
- p_\mu(z) \ge \log \capa(K)
\end{equation}
for every $z \in \del P(K)$. Suppose not. Then letting 
$$
A_\de := \{z \in \C :  -p_\mu(z) - \log \capa(K) < - \de \},
$$
there must exist some $\de > 0$ such that $\del P(K) \cap A_{\de} \ne \emptyset$. We show that $A_\de \cap \supp(\mu) \ne \emptyset$. 
\medskip
The function $-p_\mu$ is upper semicontinuous, so each $A_\de$ is open. Also, subharmonicity of $-p_\mu$ implies that each component of $A_\de$ is simply connected by the sub-mean value property. We now restrict our attention to one component $A_\de^*$ of $A_\de$ that intersects $\del P(K)$.

\medskip

The set $A_\de^*$ has the same boundary in the fine topology as it does in the usual Euclidean topology (see the discussion preceding the proof). Therefore $-p_\mu(z) = \log \capa(K) -\de$ for all $z$ in the Euclidean boundary $\del A_\de^*$.

\medskip

Moreover, for any $\eta > \de$, we can apply the same argument to $A_\eta^* = \{z \in A_\de^*: -p_\mu(z) - \log \capa(K) < - \eta \}$. This gives that $\close{A_\eta^*} \sset A_\de^*$. If $p_\mu$ were harmonic on $A_\de^*$, then it would be continuous on $\close{A_\eta^*}$, and thus have an interior minimum on $A_\eta^*$, violating the minimum principle. Therefore $p_\mu$ is not harmonic on $A_\de^*$, and hence $A_\de^* \cap \supp(\mu) \neq \emptyset.$

\medskip

We can thus find an open connected set $B$ such that $\close{B} \sset A_\de$, $B \cap U(K) \ne \emptyset$, and $B \cap \supp(\mu) \ne 0$. Let $S \sset B$ be a compact set such that $\mu(\Int(S)) > 0$.
\medskip

\medskip

By construction, $S$ is contained in the unbounded component of the complement of $P(K) \smin B$. Therefore we can invoke Lemma \ref{L:move-in-roots} applied to the sets $P(K) \smin B$ and $S$. Let $\{x_{n, i} : i \le k(n)\}$ be the set of roots of $p_n$ on $S$. Since $\mu(\Int(S)) > 0$, we have that $\liminf_{n \to \infty} k(n)/n > 0$. By Lemma \ref{L:move-in-roots}, for any function $g: \N \to [1, \infty)$ with 
\begin{equation}
\label{E:gn}
\lim_{n \to \infty} \frac{1}n \log g(n) = 0,
\end{equation}
we can find a sequence of integers $\{\ell(n): n \in \N\}$ and rational functions
$$
\theta_n(z) = \prod_{i=1}^{\ell(n)} \frac{z - y_{n,i}}{z - x_{n,i}} \qquad \text{ such that } \qquad \limsup_{n \to \infty} g(n) ||\theta_n||_{P(K) \smin B}  < \frac{1}2 \quad \mathand \quad \ell(n) = o(n).
$$
In the above formula, recall that $||\cdot||_{P(K) \smin B}$ is the uniform norm on $P(K) \smin B$. We will apply this to the function 
$$
g(n) = \frac{||p_n||_{L^p(\tau)}}{||q_n||_{L^p(\tau)}},
$$ 
where $q_n$ is the normalized $L^p(\tau)$-minimal polynomial of degree $n$ (see Example \ref{E:Lp-extremal}). The function $g$ satisfies condition \eqref{E:gn} by assumption (i) of Theorem \ref{T:det-cvg} and the regularity of $\tau$. 

\medskip

Now set
$
r_n = \theta_np_n,
$
and let $\nu_n$ be uniform measure on the roots of $r_n$. We will show that the functions $r_n$ contradict the minimality of the sequence $\{q_n\}$, therefore allowing us to conclude the inequality \eqref{E:on-bry}.

\medskip

First, since the zeros of $r_n$ and $p_n$ are the same up to a set of size $\ell(n) = o(n)$, we have that
$$
\nu_n \to \mu
$$
weakly as $n \to \infty$. Also, because $\close{B}$ is contained in $A_\de$, Lemma \ref{L:p-descent}  implies that
$$
||r_n||_B \le [\capa(K)]^{n}e^{ - \de n/2}
$$
for large enough $n$.
We now split into two cases, depending on whether $p \in (0, \infty)$, or $p = \infty$. First assume that $p = \infty$. Using the bound on $\theta_n$, for all large enough $n$ we have that
\begin{align*}
||r_n||_{L^\infty(\tau)} = ||r_n||_K &\le \max (||r_n||_B, ||r_n||_{P(K) \smin B}) \\
&\le  \max \lf([\capa(K)]^{n} e^{- \de n/2}, ||\theta_n||_{P(K) \smin B} ||p_n||_{P(K) \smin B} \rg) \\
&< ||q_n||_K.
\end{align*}
In the final line, we have used that $||q_n||_K \ge [\capa(K)]^{n}$ (see Theorem III.15 in \cite{tsuji1959potential}). As $r_n$ is a monic polynomial of degree $n$, this contradicts the minimality of $q_n$.

\medskip

The case when $p \in (0, \infty)$ is similar. For all large enough $n$ we have that
\begin{align*}
||r_n||^p_{L^p(\tau)} &= \int_B |r_n|^p d\mu + \int_{P(K) \smin B} |r_n|^p d\mu \\
&< \tau(B)||r_n||^p_B  + ||\theta_n||^p_{P(K) \smin B} \int_{P(K) \smin B} |r_n|^p d \tau \\
& \le [\capa(K)]^{np}e^{- \de n p/2} + \frac{1}{2^p} ||q_n||^p_{L^p(\tau)} < ||q_n||^p_{L^p(\tau)}.
\end{align*}
Again, this contradicts the minimality of $q_n$, and so we can conclude \eqref{E:on-bry}.

\medskip

Now, the inequality \eqref{E:on-bry} and the characterization of $p_{\mu_K}$ in Theorem \ref{T:eq-def} imply that $p_{\mu_K} - p_{\mu} \ge 0$ for quasi-every $z \in \del P(K)$. We now think of $p_{\mu_K} - p_{\mu}$ as a function on $\close{U(K)} \cup \{\infty\}$. This function is harmonic on the interior $U(K) \cup \{\infty\}$. Moreover, it is continuous at the boundary $\del P(K)$ since subharmonic functions are continuous in the fine topology, and the Euclidean and fine boundaries of connected open sets coincide. Therefore by a variant of the minimum principle for harmonic functions, (see Theorem III.28 in \cite{tsuji1959potential}), we have that
\begin{equation}
\label{E:on-bry-2}
p_{\mu_K} - p_\mu \ge 0
\end{equation}
for all $z \in U(K)$. Also, 
$$
\lim_{z \to \infty} [p_{\mu_K}(z) - p_\mu(z)] = 0.
$$
Therefore by \eqref{E:on-bry-2} and the minimum principle again, $p_{\mu_K} - p_\mu = 0$ on $U(K)$, and hence also on $\close{U(K)}$ by continuity. Moreover, $\mu_K$ is supported on $\del P(K) = \del U(K) \sset \close{U(K)}$, so the principle of domination from potential theory (see Appendix A.III, \cite{stahl1992general}) implies that $p_\mu(z) \ge p_{\mu_K}(z)$ for all $z \in \C$.
\end{proof}

\section*{Acknowledgements}
I would like to thank Thomas Bloom for many valuable conversations about the problem, and for instructive comments about previous drafts. I would also like to thank Vilmos Totik for providing the original proof of Theorem \ref{T:non-leaders}.

\bibliographystyle{acm}
\bibliography{RandomPolyBib}

\end{document}